\documentclass[12pt,a4paper]{amsart}
\usepackage[T1]{fontenc}
\usepackage{iftex}
\ifPDFTeX
  \usepackage[utf8]{inputenc}
\fi
\usepackage{newtxtext}
\usepackage{amsmath,amsthm,amssymb,physics}
\usepackage{newtxmath}
\usepackage[margin=28mm,headheight=14pt]{geometry}
\usepackage{microtype}
\usepackage{enumitem}
\usepackage{aliascnt,needspace,etoolbox}
\usepackage{tikz-cd}
\usepackage{xcolor}
\usepackage[color,matrix,arrow]{xy}
\definecolor{linkblue}{RGB}{30,65,110}
\usepackage[colorlinks=true,linkcolor=linkblue,citecolor=linkblue,urlcolor=linkblue,
  pdftitle={A Parity Obstruction to Completeness of Object Cotorsion Pairs},
  pdfauthor={Junpeng Ren and Yucheng Wang},pdfsubject={Question 29 in Ideal approximation theory}]{hyperref}
\usepackage[capitalise,noabbrev]{cleveref}
\numberwithin{equation}{section}

\newtheorem{theorem}{Theorem}[section]
\newaliascnt{question}{theorem}
\newtheorem{question}[question]{Question}
\aliascntresetthe{question}
\newaliascnt{proposition}{theorem}
\newtheorem{proposition}[proposition]{Proposition}
\aliascntresetthe{proposition}
\newaliascnt{lemma}{theorem}
\newtheorem{lemma}[lemma]{Lemma}
\aliascntresetthe{lemma}
\newaliascnt{corollary}{theorem}
\newtheorem{corollary}[corollary]{Corollary}
\aliascntresetthe{corollary}
\theoremstyle{definition}
\newaliascnt{definition}{theorem}
\newtheorem{definition}[definition]{Definition}
\aliascntresetthe{definition}
\newaliascnt{example}{theorem}

\aliascntresetthe{example}
\newaliascnt{notation}{theorem}
\newtheorem{notation}[notation]{Notation}
\aliascntresetthe{notation}
\theoremstyle{remark}
\newaliascnt{remark}{theorem}
\newtheorem{remark}[remark]{Remark}
\aliascntresetthe{remark}
\crefname{theorem}{Theorem}{Theorems}
\crefname{question}{Question}{Questions}
\crefname{proposition}{Proposition}{Propositions}
\crefname{lemma}{Lemma}{Lemmas}
\crefname{corollary}{Corollary}{Corollaries}
\crefname{definition}{Definition}{Definitions}
\crefname{example}{Example}{Examples}
\crefname{notation}{Notation}{Notations}
\crefname{remark}{Remark}{Remarks}
\BeforeBeginEnvironment{lemma}{\Needspace{7\baselineskip}}
\BeforeBeginEnvironment{proposition}{\Needspace{8\baselineskip}}
\BeforeBeginEnvironment{corollary}{\Needspace{7\baselineskip}}
\BeforeBeginEnvironment{definition}{\Needspace{10\baselineskip}}
\BeforeBeginEnvironment{example}{\Needspace{7\baselineskip}}

\newcommand{\A}{\mathcal A}
\newcommand{\D}{\mathcal D}
\newcommand{\F}{\mathcal F}
\newcommand{\C}{\mathcal C}
\newcommand{\I}{\mathcal I}
\newcommand{\J}{\mathcal J}
\newcommand{\E}{\mathcal E}
\newcommand{\X}{\mathcal X}
\newcommand{\Y}{\mathcal Y}
\newcommand{\Hom}{\operatorname{Hom}}
\newcommand{\Ext}{\operatorname{Ext}^{1}}
\newcommand{\Ob}{\operatorname{Ob}}
\newcommand{\im}{\operatorname{im}}
\newcommand{\Ch}{\operatorname{Ch}}
\newcommand{\K}{\operatorname{K}}
\newcommand{\Vect}{\operatorname{vect}}
\newcommand{\Cone}{\operatorname{Cone}}
\newcommand{\smd}{\operatorname{smd}}
\newcommand{\stalk}[2]{S^{#1}(#2)}
\newcommand{\Ho}{\operatorname{H}}
\newcommand{\stab}{\underline{\A}}
\newcommand{\sHom}{\underline{\Hom}_{\A}}
\newcommand{\ul}[1]{\underline{#1}}
\newcommand{\T}{\mathcal T}
\newcommand{\sC}{\underline{\C}}
\newcommand{\sHomC}{\underline{\Hom}_{\C}}

\setlist[enumerate]{label=\textup{(\roman*)},leftmargin=2em,itemsep=3pt,topsep=5pt}
\title[A parity obstruction to completeness]{A parity obstruction to completeness of object cotorsion pairs}

\author{Junpeng Ren}
\thanks{Junpeng Ren was supported by the China Scholarship Council (Grant No.~202606620154). Yucheng Wang was supported by the China Scholarship Council (Grant No.
~202506190133).}
\address{School of Mathematics and Statistics, Northeast Normal University, Changchun 130024, China}
\email{renjp@nenu.edu.cn}
\author{Yucheng Wang\textsuperscript{*}}
\thanks{\textsuperscript{*}Corresponding author: Yucheng Wang (\texttt{wangyucheng2358@163.com}).}
\address{School of Mathematics, Nanjing University, Nanjing 210093, Jiangsu Province, P.R. China}
\email{wangyucheng2358@163.com}
\date{September 2026}


\begin{document}

\begin{abstract}
	Fu, Guil Asensio, Herzog and Torrecillas asked whether a complete
	ideal cotorsion pair of object ideals in an exact category always
	induces a complete cotorsion pair of objects.
	We show that it need not, even in a Hom-finite, weakly idempotent
	complete Frobenius exact category.
	Our example consists of bounded complexes of finite-dimensional
	vector spaces with even total cohomology dimension.
	We characterize special object approximations by parity and
	identify the obstruction as a non-split idempotent in the stable
	category.
	For cohereditary pairs, a criterion of Saor\'{\i}n and
	\v{S}\v{t}ov\'{\i}\v{c}ek gives an affirmative answer when the
	stable category is idempotent complete.
\end{abstract}
\maketitle
\markboth{A PARITY OBSTRUCTION TO COMPLETENESS}{A PARITY OBSTRUCTION TO COMPLETENESS}

\section{Introduction}

Classical approximation theory, going back to Auslander and
Smal{\o} \cite{AS80} and Enochs \cite{Enochs81}, approximates
objects in categories of $R$-modules by objects from a fixed class. 
Cotorsion pairs, introduced
by Salce \cite{Salce79}, are important tools to study approximation theory.
(Complete) cotorsion pairs provide (special) precovers and (special)
preenvelopes \cite{GT}. Ideal approximation theory is a natural generalization of
classical approximation theory. It
grew out of Herzog's work on phantom covers \cite{Herzog07}
and extends this framework from classes of objects to ideals of
morphisms. Its foundations in exact categories were laid by Fu,
Guil Asensio, Herzog and Torrecillas \cite{FGHT}, and subsequent
developments include powers of the phantom ideal \cite{FH}
and transfinite powers of ghost ideals \cite{EFHO}.
The two theories meet in object ideals.

Let $(\A,\E)$ be an exact category. For an additive class of objects $\F$, i.e., a class closed under finite direct sums 
and direct summands in $\A$,
write $\langle\F\rangle$ for the ideal of morphisms factoring through
an object of $\F$. For an ideal $\I$, put
\[ \Ob(\I)=\{X\in\A:1_X\in\I\}. \]
The ideal $\I$ is an \emph{object ideal} if
$\I=\langle\Ob(\I)\rangle$.
By \cite[Theorem~28]{FGHT}, a complete cotorsion pair of objects
$(\F,\T)$ gives a complete ideal cotorsion pair
$(\langle\F\rangle,\langle\T\rangle)$.
The converse was posed as an open question.

\begin{question}[{\cite[Question~29]{FGHT}}]\label{qu:original}
Let $(\I,\J)$ be a complete ideal cotorsion pair in an exact category. 
If both $\I$ and $\J$ are object ideals, is the cotorsion pair $(\Ob(\I),\Ob(\J))$ complete?
\end{question}

Several partial answers are known. The answer is affirmative for exact
categories with enough projective and injective objects and an enveloping
ideal $\J$ \cite[Theorem~3.6(2)]{SWZ}, for Frobenius categories with
a Krull--Schmidt stable category \cite[Corollary~1.3]{STWZ}, and for
Krull--Schmidt exact categories
\cite[Theorem~3.11 and Corollary~3.12]{ZZ}.
Wang, Wang and Zhu \cite[Theorem~1 and Example~2]{WWZ} gave a negative
answer in an exact category that is not weakly idempotent complete.

We show that the answer remains negative even for Hom-finite, weakly
idempotent complete Frobenius exact categories, and that the failure
there is as small as it can be (\cref{thm:main}). First, the ambient category 
cannot be improved: Hom-finite idempotent complete $k$-linear categories are Krull--Schmidt
\cite[Corollary~4.4]{Krause}, so \cite[Corollary~3.12]{ZZ} applies to
them. Hence, among Hom-finite $k$-linear exact categories, the answer
is affirmative under idempotent completeness and negative under weak
idempotent completeness. Second, the failure itself is minimal. The
left class is precovering and the right class is preenveloping, so only
\emph{special} approximations fail, and they fail already for the
single object $M=\stalk{0}{k}\oplus\stalk{2}{k}$ of total cohomology
dimension~$2$, the smallest possible (objects with $\beta=0$ are
contractible, hence projective--injective). This object lies in
$\smd(\F\oplus\T)$, the test class of the completeness criterion in
\cite[Corollary~1.3]{STWZ} (\cref{rem:STWZ}). More precisely, an object
has a special approximation exactly when the corresponding truncation
has even total cohomology dimension (\cref{prop:criterion}).

Our second aim is to explain where this parity obstruction comes from.
For cohereditary cotorsion pairs in an arbitrary Frobenius exact
category, a special ideal precover of a single object $A$ produces an
idempotent in the stable category, whether it splits does not depend on
the choices made. It splits if and only if $A$ has a special object
precover (\cref{prop:t-type-positive}). Special preenvelopes are
handled by passing to $A[1]$. No idempotent completeness is assumed.
In our example the idempotent is the projection onto the cohomology in
degrees $\ge2$, and its splitting is precisely the parity condition
(\cref{cor:parity-idempotent}). Thus the failure of object completeness
is exactly the failure of a certain idempotent to split. In
particular, when the stable category is idempotent complete,
completeness passes from a complete ideal cotorsion pair of object
ideals to its cohereditary object pair. This also follows from a
criterion of Saor\'{\i}n and \v{S}\v{t}ov\'{\i}\v{c}ek
\cite[Proposition~3.11]{SS}, and by \cite[Corollary~1.2]{STWZ} it
suffices that $\I$ be precovering (\cref{cor:positive}). For
cohereditary pairs this replaces the Krull--Schmidt hypothesis of
\cite[Corollary~1.3]{STWZ} by idempotent completeness, and our example,
which is cohereditary, shows that this hypothesis cannot be dropped.

\medskip

\noindent\textbf{The example.}
Fix a field $k$, let $\Vect_k$ be the category of finite-dimensional $k$-vector spaces, and let $\D=\Ch^b(\Vect_k)$ be the abelian category of bounded cochain complexes. For $X\in\D$ write
\[
 \beta(X)=\sum_{n\in\mathbb Z}\dim_k H^n(X)
\]
for its total cohomology dimension, and put
\[
 \A=\{X\in\D:\beta(X)\in2\mathbb Z\},
\]
with the short exact sequences of $\D$ whose terms lie in $\A$ as conflations. Define
\begin{align*}
 \F&=\{X\in\A:H^n(X)=0\text{ for all }n\ge1\},\\
 \T&=\{X\in\A:H^n(X)=0\text{ for all }n\le1\}.
\end{align*}

\begin{theorem}\label{thm:main}
The category $\A$ is a Hom-finite, weakly idempotent complete Frobenius exact category that is not idempotent complete. The pair
$
 (\langle\F\rangle,\langle\T\rangle)
$
is a complete ideal cotorsion pair of object ideals, with $\Ob(\langle\F\rangle)=\F$ and $\Ob(\langle\T\rangle)=\T$. 
The cotorsion pair of objects $(\F,\T)$ is neither special precovering nor special preenveloping. 
In fact, the object
\[
 M=\stalk{0}{k}\oplus\stalk{2}{k}\in\A
\]
has neither a special $\F$-precover nor a special $\T$-preenvelope.
\end{theorem}

The category $\A$ omits certain direct summands, such as $\stalk{0}{k}$, while retaining their doubles. 
Thus the obstruction is the absence of idempotent summands, not the absence of kernels of retractions (\cref{sec:stable}).

\medskip
\noindent\textbf{Organisation.} \Cref{sec:prelim} recalls ideal cotorsion pairs and the criterion for
special ideal approximations. \Cref{sec:construction} constructs the
Frobenius category
 and establishes
the complete ideal cotorsion pair. In \cref{sec:failure},
\cref{sec:object-failure} proves the parity criterion and
\cref{thm:main}. \cref{sec:stable} recalls how special object
approximations in a Frobenius category correspond to triangles in its
stable category. Especially, \cref{cor:positive} gives the positive result for
cohereditary pairs. We then construct a stable idempotent
(\cref{prop:t-type-positive}) and compute it in the example
(\cref{cor:parity-idempotent}), before ending with
\cref{qu:idempotent-complete}.

\section{Preliminaries}\label{sec:prelim}

Throughout this section $(\A,\E)$ is an arbitrary exact category in the sense of Quillen; see \cite{Buhler}. 
Recall from \cite{Buhler} that an exact category
$(\A,\E)$ consists of an additive category $\A$ and an isomorphism-closed class $\E$ of
distinguished kernel-cokernel pairs $(i,p)$ in $\A$ \[\xi :A\stackrel{i}\to B
\stackrel{p}\to C,\] where $i$ and $p$ are called the {\it inflation} and the {\it deflation} of
$\xi$, respectively, satisfying:

{\rm (i)} For every object $A\in\A$, the identity morphism $1_{A}$ is both an inflation and a deflation;

{\rm (ii)} Inflations and deflations are closed under composition;

{\rm (iii)} The class $\E$ is closed under pushout and pullback along any morphism in $\A$.

A kernel-cokernel pair $(i,p)$, or $\xi$, in $\E$ is called a {\it conflation} in
$(\A,\E)$. We will also sometimes say that $C$ is the cokernel of $i$ and
use the notation $C=B/A$; $A$ is the kernel of $p$, which may be denoted by $A\subseteq B$.
For given objects $M$ and $N$ in $\A$, the axioms of exact categories allow
for the usual construction of the Yoneda Ext bifunctor $\Ext_{\A}(M,N)$. It is the
abelian group of equivalence classes of conflations $N\rightarrowtail Z\twoheadrightarrow M$
with respect to the {\it Baer sum} operation. In particular, we obtain that
$\Ext_{\A}(M,N)=0$ iff every conflation $N\rightarrowtail Z\twoheadrightarrow M$ is
isomorphic to $N\rightarrowtail N\oplus M\twoheadrightarrow M$. In general, we say a conflation
$\xi:A\stackrel{i}\to B \stackrel{p}\to C$ in $\E$ is {\it trivial} if it represents
$0$, that is, the morphism $i$ is a section: there exists a morphism $r:B\to A$ such that
$ri=1_A$; or equivalently, $p$ is a retraction: there exists a morphism $s:C\to B$ such
that $ps=1_C$; see \cite[Remark 7.4]{Buhler}.
We say an object $I\in\A$ is {\it injective} if any inflation $I\rightarrowtail Z$
splits. It follows that $I$ is injective if and only if$\Ext_{\A}(N,I)=0$ for
any $N\in\A$, or equivalently, $\Hom_{\A}(-,I)$ is exact.
{\it Projective objects} in $\A$ are defined dually. An exact category $\A$
is said to have {\it enough projectives} if for every object $M\in\A$ there exists a
projective object $P$ and a deflation $P\to M$, see \cite[Definition 11.9]{Buhler}. The notion
of having {\it enough injectives} is defined dually. 
An exact category is said to be \emph{Frobenius} if it has enough projectives and injectives and, moreover, 
the classes of projectives and injectives coincide (\cite[Section 13.4]{Buhler} and \cite{Heller}).
An exact category is said to be \emph{idempotent complete} if for every idempotent $p:A\to A$, 
i.e., $p^2=p$, there is a decomposition $A\cong K\oplus I$ of $A$ such that 
$p\cong \left[\smqty{0&0\\0&1}\right]$, see \cite[Section 6]{Buhler}.
An exact category is said to be {\it weakly idempotent complete} if every
retraction has a kernel, see \cite[Section 7]{Buhler}.

An \emph{ideal} of $\A$ is an additive subbifunctor $\I$ of $\Hom_{\A}(-,-)$; equivalently, 
a class of morphisms closed under addition and under composition with arbitrary morphisms on either side. 
For morphisms $f:X_0\to X_1$ and $g:Y_0\to Y_1$, the map
\[
 \Ext_{\A}(f,g):\Ext_{\A}(X_1,Y_0)\longrightarrow\Ext_{\A}(X_0,Y_1)
\]
is pullback along $f$ followed by pushout along $g$ (the order does not matter). For ideals $\I$ and $\J$ put
\[
 {}^\perp\J=\{f:\Ext_{\A}(f,g)=0\text{ for all }g\in\J\},\qquad
 \I^\perp=\{g:\Ext_{\A}(f,g)=0\text{ for all }f\in\I\}.
\]
The pair $(\I,\J)$ is an \emph{ideal cotorsion pair} if $\I={}^\perp\J$ and $\J=\I^\perp$. For classes of objects 
$\F$ and $\T$, the pair $(\F,\T)$ is a \emph{cotorsion pair} if $\F={}^\perp\T$ and $\T=\F^\perp$, 
where orthogonality means vanishing of the extension groups.

For a cotorsion pair of objects $(\F,\T)$ and any object $A$ in $\A$, a \emph{special $\F$-precover} of $A$ is 
a conflation $0\to T\to F\to A\to0$ with $F\in\F$ and $T\in\T$, and a \emph{special $\T$-preenvelope} of $A$ is a conflation 
$0\to A\to T\to F\to0$ with $T\in\T$ and $F\in\F$. The pair is \emph{complete} if every object in $\A$ has both.

\begin{definition}[{\cite{FGHT}}]
Let $\I$ be an ideal and $A\in\A$. An \emph{$\I$-precover} of $A$ is a morphism $p:E\to A$ in $\I$ such that 
every morphism in $\I$ with codomain $A$ factors through $p$. It is \emph{special} if it is obtained as the pushout 
of a conflation along a morphism in $\I^\perp$:
\[
\begin{tikzcd}[column sep=large]
0\ar[r]&Y\ar[r]\ar[d,"g"']&Z\ar[r]\ar[d]&A\ar[r]\ar[d,equal]&0\\
0\ar[r]&B\ar[r]&E\ar[r,"p"']&A\ar[r]&0
\end{tikzcd}
\qquad g\in\I^\perp.
\]
Special $\J$-preenvelopes are defined dually, as pullbacks of conflations along morphisms in ${}^\perp\J$. An ideal cotorsion pair $(\I,\J)$ is \emph{complete} if every object has a special $\I$-precover and a special $\J$-preenvelope.
\end{definition}

The following criterion is all we need to produce special ideal approximations; compare \cite[Proposition~11 and the discussion before Proposition~25]{FGHT}, where such approximations are called object-special.

\begin{lemma}\label{lem:object-special}
Let $(\I,\J)$ be an ideal cotorsion pair.
\begin{enumerate}
\item If $0\to T\to E\xrightarrow{p}A\to0$ is a conflation with $p\in\I$ and $1_{T}\in\J$, then $p$ is a special $\I$-precover of $A$.
\item If $0\to A\xrightarrow{j}D\to F\to0$ is a conflation with $j\in\J$ and $1_F\in\I$, then $j$ is a special $\J$-preenvelope of $A$.
\end{enumerate}
\end{lemma}
\begin{proof}
(i) Let $\eta\in\Ext_{\A}(A,T)$ be the class of the conflation. 
Since $1_{T}\in\J=\I^\perp$, we have $i^*\eta=0$ for every $i:X\to A$ in $\I$. 
Hence the pullback of $\eta$ along $i$ splits, that is, $i$ factors through $p$. Since $p\in\I$, it is an $\I$-precover. 
It is special because the conflation is its own pushout along $1_{T}\in\I^\perp$.

The proof of (ii) is dual.\qedhere

\end{proof}
For object ideals, an ideal cotorsion pair always induces a cotorsion pair of objects, 
so \cref{qu:original} is only about completeness.

\begin{lemma}[{\cite[Lemma~3.5]{SWZ}}]\label{lem:object-pair}
An ideal cotorsion pair $(\I,\J)$ of object ideals induces the cotorsion pair of objects $(\Ob(\I),\Ob(\J))$.
\end{lemma}

\section{The construction}\label{sec:construction}

\subsection{The Frobenius category}\label{sec:category}

\begin{notation}\label{not:conventions}
Let $\D=\Ch^b(\Vect_k)$ and let $\K=\K^b(\Vect_k)$ be its homotopy category. Every short exact sequence in $\D$ is degreewise split.
\begin{itemize}[leftmargin=1.5em]
\item For $V\in\Vect_k$ and $r\in\mathbb Z$, $\stalk{r}{V}$ is the stalk complex with $V$ in degree $r$.
\item For $s\in\mathbb Z$, the shift is $X[s]^n=X^{n+s}$ with $d_{X[s]}=(-1)^sd_X$. Thus $H^n(X[s])=H^{n+s}(X)$ and $X[s][t]=X[s+t]$.
\item $\Ho(X)$ is the complex with $\Ho(X)^n=H^n(X)$ and zero differential. For $a\in\mathbb Z$, $\Ho^{\le a}(X)$ and $\Ho^{\ge a}(X)$ are the subcomplexes of $\Ho(X)$ concentrated in degrees $\le a$ and $\ge a$, respectively.
\item For a chain map $f:X\to Y$, the cone is $\Cone(f)^n=Y^n\oplus X^{n+1}$ with $d(y,x)=(d_Yy+f(x),-d_Xx)$.
\item A complex $X$ is \emph{contractible} if $1_X$ is null-homotopic.
\item $\beta(X)=\sum_n\dim_kH^n(X)$ and $\chi(X)=\sum_n(-1)^n\dim_kH^n(X)$.
\end{itemize}
\end{notation}

\begin{lemma}\label{lem:euler}
For $X\in\D$,
\[
 \chi(X)=\sum_n(-1)^n\dim_kX^n,\qquad \chi(X)\equiv\beta(X)\pmod 2,\qquad \chi(X[s])=(-1)^s\chi(X).
\]
Moreover $\chi$ is additive on short exact sequences in $\D$. Consequently, $X\in\A$ if and only if $\chi(X)$ is even.
\end{lemma}
\begin{proof}
Put $Z^n=\ker d_X^n$ and $B^n=\im d_X^{n-1}$. Since $d_X^nd_X^{n-1}=0$, we have $B^n\subseteq Z^n$ and $H^n(X)=Z^n/B^n$. The short exact sequences
\begin{gather*}
 0\longrightarrow Z^n\longrightarrow X^n\xrightarrow{d_X^n}B^{n+1}\longrightarrow0, \ \ \
 0\longrightarrow B^n\longrightarrow Z^n\longrightarrow H^n(X)\longrightarrow0
\end{gather*}
give
\[
 \dim_k X^n=\dim_k Z^n+\dim_k B^{n+1}
 =\dim_k H^n(X)+\dim_k B^n+\dim_k B^{n+1}.
\]
All sums are finite because $X$ is bounded, so
\begin{align*}
 \sum_n(-1)^n\dim_k X^n
 &=\sum_n(-1)^n\bigl(\dim_k H^n(X)+\dim_k B^n+\dim_k B^{n+1}\bigr)\\
 &=\sum_n(-1)^n\dim_k H^n(X)=\chi(X).
\end{align*}
The congruence follows from $(-1)^n\equiv1\pmod2$. For the shift formula, reindexing by $m=n+s$ gives
\[
 \chi(X[s])=\sum_n(-1)^n\dim_k H^{n+s}(X)
 =(-1)^s\chi(X).
\]
Finally, a short exact sequence $0\to X\to Y\to Z\to0$ in $\D$ gives $\dim_k Y^n=\dim_k X^n+\dim_k Z^n$ in each degree. Taking alternating sums proves $\chi(Y)=\chi(X)+\chi(Z)$.
\end{proof}

\begin{proposition}\label{prop:category}
The full subcategory $\A$ of $\D$, with the short exact sequences of $\D$ whose terms lie in $\A$ as conflations, 
is an essentially small, Hom-finite exact $k$-category. It is weakly idempotent complete but not idempotent complete.
\end{proposition}
\begin{proof}
By \cref{lem:euler}, $\A$ contains $0$, is closed under finite direct sums, 
and is closed under extensions in the abelian category $\D$. 
An extension-closed full subcategory of an abelian category is exact with the induced conflations by \cite[Lemma~10.20]{Buhler}. 
Essential smallness and Hom-finiteness are inherited from $\D$.

Let $i:X\to Y$ be a split monomorphism in $\A$, and let $Z$ be its cokernel in $\D$. 
Then $Y\cong X\oplus Z$, so $\chi(Z)=\chi(Y)-\chi(X)$ is even and $Z\in\A$. It follows that $i$ has a cokernel in $\A$. 
Thus every split monomorphism in $\A$ has a cokernel, so every retraction has a kernel by \cite[Lemma~7.1]{Buhler}, 
and $\A$ is weakly idempotent complete.


Let $S=\stalk{0}{k^2}\in\A$ and let $e:S\to S$ be the projection $(x,y)\mapsto(x,0)$ in degree $0$. 
If $e$ split in $\A$, there would be an object $B\in\A$ and chain maps
\[
 B\xrightarrow{i}S\xrightarrow{p}B,
 \qquad pi=1_B,\quad ip=e.
\]
The equality $pi=1_B$ makes $i$ degreewise injective, while $ei=i$ and $e=ip$ give $\im i=\im e$. Hence, in $\D$,
\[
 B\cong\im e=\stalk{0}{k\oplus0}\cong\stalk{0}{k}.
\]
Thus $\beta(B)=1$, contradicting $B\in\A$. Therefore $\A$ is not idempotent complete.
\end{proof}

The following cone sequences give enough projective and injective objects (\cref{prop:frobenius}) 
and are the building blocks of the approximations in \cref{sec:pair}.

For $U\in\D$, write $\iota_U(a)=(a,0)$ and $\pi_U(a,b)=b$, and put
\[
 P(U)=\Cone(1_{U[-1]}),\qquad q_U=\pi_{U[-1]}:P(U)\longrightarrow U[-1][1]=U.
\]
With these conventions, we have short exact sequences
\begin{equation}\label{eq:cones}
 0\to U\xrightarrow{\iota_U}\Cone(1_U)\xrightarrow{\pi_U}U[1]\to0,\qquad
 0\to U[-1]\to P(U)\xrightarrow{q_U}U\to0.
\end{equation}
The cone $\Cone(1_U)$ is contractible, with contracting homotopy $h^n(a,b)=(0,a)$, and the same holds for $P(U)$. 
Thus the middle terms of \eqref{eq:cones} have zero cohomology and lie in $\A$, even when $U\notin\A$.

\begin{proposition}\label{prop:frobenius}
The exact category $\A$ is Frobenius, and its projective objects and injective objects are exactly the 
contractible complexes in $\A$.
\end{proposition}
\begin{proof}
\looseness=-1
With the degreewise split exact structure, the category $\D$ is Frobenius, and its projective--injective 
objects are precisely the contractible complexes by \cite[Section~9.2]{Happel87}. Note that the projective objects 
in $\A$ are precisely the projective objects in $\D$ since $\chi(X)=0$ for any contractible complex $X$. 
It follows that the exact category $\A$ is Frobenius.
\end{proof}

\subsection{Cohomology and extensions}\label{sec:ext}

To compute extensions in $\A$, we first reduce complexes and homotopy classes of chain maps to cohomology.

\begin{lemma}[{\cite[Exercise~1.1.3]{Weibel}}]\label{lem:splitting}
Every $X\in\D$ admits an isomorphism of complexes $\Phi_X:X\xrightarrow{\sim}\Ho(X)\oplus Q_X$ with $Q_X$ contractible, 
such that $H^n(\Phi_X)$ is the canonical identification $H^n(X)=H^n(\Ho(X)\oplus Q_X)$ for every $n$. 
Moreover $\chi(Q_X)=0$, so if $X\in\A$ then $\Ho(X)$ and $Q_X$ lie in $\A$.
\end{lemma}
\begin{proof}
Put $Z^n=\ker d_X^n$ and $B^n=\im d_X^{n-1}$, and choose complements $Z^n=B^n\oplus L^n$ and $X^n=Z^n\oplus T^n$. 
The quotient map restricts to an isomorphism $L^n\cong H^n(X)$, and $\delta^n=d_X^n|_{T^n}:T^n\to B^{n+1}$ is an isomorphism. 
In the decomposition $X^n=B^n\oplus L^n\oplus T^n$ the differential is $(b,\ell,t)\mapsto(\delta^nt,0,0)$. 
Let $Q_X$ be the complex $Q_X^n=B^n\oplus T^n$ with $d(b,t)=(\delta^nt,0)$. Then
\[
 \Phi_X^n(b,\ell,t)=\bigl([\ell],(b,t)\bigr)
\]
is an isomorphism of complexes $X\to\Ho(X)\oplus Q_X$ that induces the canonical identification on cohomology. 
The maps $s^n(b,t)=(0,(\delta^{n-1})^{-1}b)$ satisfy $ds+sd=1_{Q_X}$, so $Q_X$ is contractible. Finally, 
$H^n(Q_X)=0$ for every $n$, so $\chi(Q_X)=0$, and $\chi(\Ho(X))=\chi(X)$.
\end{proof}

\begin{lemma}[{\cite[Example~6.3.2]{Riehl}}]\label{lem:homotopy}
For $X,Y\in\D$, the map
\[
 \Hom_{\K}(X,Y)\longrightarrow\prod_n\Hom_k\bigl(H^n(X),H^n(Y)\bigr),\qquad [f]\longmapsto\bigl(H^n(f)\bigr)_n,
\]
is an isomorphism. In particular, a chain map inducing zero on all cohomology groups is null-homotopic, 
and hence factors through $\Cone(1_X)$.
\end{lemma}
\begin{proof}
The map is well defined and additive. By \cref{lem:splitting}, conjugating with $\Phi_X$ and $\Phi_Y$ reduces the 
claim to $X=\Ho(X)\oplus Q_X$ and $Y=\Ho(Y)\oplus Q_Y$, since $\Phi_X,\Phi_Y$ induce the canonical identifications 
on cohomology. The contractible summands are zero objects of $\K$ and have zero cohomology, so we may further 
assume that $X$ and $Y$ have zero differential. Then chain maps are arbitrary families of linear maps $X^n\to Y^n$, 
every homotopy $h$ gives $d_Yh+hd_X=0$, and $H^n(f)=f^n$. For the final assertion, if $r=d_Yh+hd_X$, then $r=\varphi\iota_X$, 
where the chain map $\varphi:\Cone(1_X)\to Y$ is given by $\varphi^n(a,b)=r^n(a)+h^{n+1}(b)$.
\end{proof}

We now realize the stable category of $\A$ as a full triangulated
subcategory of $\K\simeq D^b(\Vect_k)$.

Let $\mathcal P$ be the class of contractible complexes. A morphism in $\D$
factors through $\mathcal P$ if and only if it is null-homotopic
(cf.\ \cref{lem:homotopy}), so $\D/\mathcal P=\K$; by \cref{prop:frobenius},
$\mathcal P$ is also the class of projective objects of $\A$.
Hence the inclusion $\A\subseteq\D$ induces a fully faithful embedding
\[
 \stab=\A/\mathcal P\lhook\joinrel\longrightarrow\D/\mathcal P=\K
 \simeq D^b(\Vect_k),
\]
where $\K\simeq D^b(\Vect_k)$ because every acyclic complex in $\D$ is
contractible by \cref{lem:splitting}. By the definition of $\A$, the
essential image of $\stab\to\K$ consists precisely of the objects
of even total cohomology dimension. We therefore obtain an equivalence
\begin{equation}\label{eq:stable-category}
 \stab\simeq\mathcal S:=\{X\in\K:\beta(X)\text{ is even}\}.
\end{equation}
This equivalence is triangulated: the cone sequences \eqref{eq:cones}
identify the stable suspension with $[1]$, and conflations induce the
standard triangles in $\K$;
see \cite[Chapter~I, Theorem~2.6 and Lemmas~2.7--2.8]{Happel}.

The subcategory $\mathcal S$ is dense in $\K$, since every $X\in\K$
is a direct summand of $X\oplus X\in\mathcal S$. Under the isomorphism
$K_0(\K)\cong\mathbb Z$ induced by $\chi$, it corresponds to
$2\mathbb Z$ in Thomason's classification
\cite[Theorem~2.1]{Thomason}.
We use the same notation for an object of $\A$ and its image in $\K$.

\begin{proposition}\label{prop:ext}
For $X,Y\in\A$ there are isomorphisms of $k$-vector spaces, natural in both variables,
\begin{equation}\label{eq:ext-formula}
 \Ext_{\A}(X,Y)\cong\Hom_{\K}(X,Y[1])\cong\bigoplus_{n\in\mathbb Z}\Hom_k\bigl(H^n(X),H^{n+1}(Y)\bigr).
\end{equation}
Under these isomorphisms, for $f:X_0\to X_1$ and $g:Y_0\to Y_1$ in $\A$, the map $\Ext_{\A}(f,g)$ sends a family $(u_n)_n$ to
\begin{equation}\label{eq:ext-action}
 \bigl(H^{n+1}(g)\,u_n\,H^n(f)\bigr)_n.
\end{equation}
\end{proposition}
\begin{proof}
For a Frobenius category, the connecting morphism of a conflation gives a natural $k$-linear isomorphism
\[
 \Ext_{\A}(X,Y)\cong\sHom(X,Y[1]);
\]
by the suspension and connecting-morphism constructions in \cite[Chapter~I, 2.2 and~2.7]{Happel}. 
By \eqref{eq:stable-category}, the right-hand side is $\Hom_{\K}(X,Y[1])$. Applying \cref{lem:homotopy} gives 
the second isomorphism in \eqref{eq:ext-formula}.

Under the first isomorphism, pullback along $f$ and pushout along $g$ send a connecting morphism 
$\alpha:X_1\to Y_0[1]$ to $g[1]\alpha f$. Applying cohomology gives
\[
 H^n(g[1]\alpha f)=H^{n+1}(g)\,H^n(\alpha)\,H^n(f),
\]
which proves \eqref{eq:ext-action} and naturality in both variables.
\end{proof}

\subsection{The complete ideal cotorsion pair}\label{sec:pair}

Recall from the introduction that we define the two classes $\F$ (cohomology in degrees $\le0$) 
and $\T$ (cohomology in degrees $\ge2$):
\begin{align*}
 \F&=\{X\in\A:H^n(X)=0\text{ for all }n\ge1\},\\
 \T&=\{X\in\A:H^n(X)=0\text{ for all }n\le1\}.
\end{align*}
Note that degree $1$ is allowed in neither class, as this is forced by the degree shift in \eqref{eq:ext-formula}.
Define ideals of $\A$ by
\begin{align*}
 \I&=\{f\in\operatorname{Mor}\A:H^n(f)=0\text{ for every }n\ge1\},\\
 \J&=\{g\in\operatorname{Mor}\A:H^n(g)=0\text{ for every }n\le1\}.
\end{align*}
Since $H^n$ is an additive functor, $\I$ and $\J$ are ideals of $\A$, and applying the defining conditions to identity morphisms gives
\begin{equation}\label{eq:object-classes}
 \Ob(\I)=\F,\qquad \Ob(\J)=\T.
\end{equation}

The truncations of an object of $\A$ need not lie in $\A$, but their doubles do, with the same cohomological support. 
This gives the factorizations through $\F$ and $\T$ needed below.

\begin{proposition}\label{prop:object-ideals}
$\I=\langle\F\rangle$ and $\J=\langle\T\rangle$. In particular, $\I$ and $\J$ are object ideals.
\end{proposition}
\begin{proof}
A morphism factoring through an object of $\F$ induces zero on $H^n$ for $n\ge1$, so $\langle\F\rangle\subseteq\I$.

Conversely, let $f:X\to Y$ lie in $\I$. By \cref{lem:splitting} there is a decomposition $X\cong L\oplus U\oplus Q_X$ in $\D$ 
with $L=\Ho^{\le0}(X)$ and $U=\Ho^{\ge1}(X)$ (the summands $L$ and $U$ need not lie in $\A$). 
Let $p_L:X\to L$ and $j_L:L\to X$ be the corresponding projection and inclusion, and put $f_0=fj_Lp_L$. The idempotent $j_Lp_L$ 
induces the identity on $H^n(X)$ for $n\le0$ and zero for $n\ge1$. Since $H^n(f)=0$ for $n\ge1$, the maps $f$ and $f_0$ 
induce the same maps on all cohomology groups. By \cref{lem:homotopy}, $r=f-f_0$ factors through $\Cone(1_X)$.

The map $f_0$ factors as
\[
 X\xrightarrow{\ \binom{p_L}{0}\ }L\oplus L\xrightarrow{\ (fj_L\ \ 0)\ }Y.
\]
Here $\chi(L\oplus L)=2\chi(L)$ is even and $L\oplus L$ has cohomology only in degrees $\le0$, so $L\oplus L\in\F$. 
The contractible complex $\Cone(1_X)$ also lies in $\F$. Adding the two factorizations, $f$ factors through 
$(L\oplus L)\oplus\Cone(1_X)\in\F$. Hence $\I\subseteq\langle\F\rangle$.

The argument for $\J$ is the same, with $W=\Ho^{\ge2}(X)$ in place of $L$ and corresponding maps $p_W:X\to W$ and 
$j_W:W\to X$: if $g:X\to Y$ lies in $\J$ and $g_0=gj_Wp_W$, then $g-g_0$ induces zero on cohomology, 
$g_0$ factors through $W\oplus W\in\T$, and $\Cone(1_X)\in\T$.
\end{proof}

By \eqref{eq:ext-action}, $\Ext_{\A}(f,g)$ is computed degree by degree on cohomology. 
This gives one inclusion in each orthogonality relation, and stalk complexes give the other.

\begin{proposition}\label{prop:orthogonal}
${}^\perp\J=\I$ and $\I^\perp=\J$. Hence $(\I,\J)$ is an ideal cotorsion pair.
\end{proposition}
\begin{proof}
Let $f\in\I$ and $g\in\J$. In \eqref{eq:ext-action}, the $n$-th component vanishes because $H^n(f)=0$ if $n\ge1$, 
and $H^{n+1}(g)=0$ if $n\le0$. Thus $\I\subseteq{}^\perp\J$ and $\J\subseteq\I^\perp$.

Let $f:X_0\to X_1$ be a morphism not in $\I$, and choose $n\ge1$ and $x\in H^n(X_0)$ with $v=H^n(f)(x)\neq0$. 
Put $T=\stalk{n+1}{k^2}$. Then $\beta(T)=2$, so $T\in\A$. It follows that $T\in\T$ and 
$1_T\in\J$ since $n+1\geq 2$.
By \cref{prop:ext} we have that $\Ext_{\A}(X_i,T)\cong\Hom_k(H^n(X_i),k^2)$ for $i=0,1$, 
and $\Ext_{\A}(f,1_T)$ corresponds to $u\mapsto uH^n(f)$. Choose a linear map $u:H^n(X_1)\to k^2$ with $u(v)\neq0$. 
Then $uH^n(f)\neq0$, so $\Ext_{\A}(f,1_T)\neq0$ and $f\notin{}^\perp\J$.

Dually, let $g:Y_0\to Y_1$ be a morphism not in $\J$, and choose $m\le1$ with $H^m(g)\neq0$. 
Put $F'=\stalk{m-1}{k^2}$; then $F'\in\F$ since $m-1\le0$, 
so $1_{F'}\in\I$. Now $\Ext_{\A}(F',Y_i)\cong\Hom_k(k^2,H^m(Y_i))$ 
and $\Ext_{\A}(1_{F'},g)$ corresponds to $u\mapsto H^m(g)u$. Choosing $u$ whose image 
contains a vector not killed by $H^m(g)$ gives $\Ext_{\A}(1_{F'},g)\neq0$, so $g\notin\I^\perp$.
\end{proof}

\begin{corollary}\label{cor:pair}
The pair $(\I,\J)$ is an ideal cotorsion pair of object ideals with $\Ob(\I)=\F$ and 
$\Ob(\J)=\T$, and $(\F,\T)$ is a cotorsion pair of objects in $\A$.
\end{corollary}
\begin{proof}
Combine \eqref{eq:object-classes}, \cref{prop:object-ideals,prop:orthogonal,lem:object-pair}.
\end{proof}

We now prove completeness. All auxiliary complexes $L,U,V,W$ below are formed in $\D$, and membership in $\A$ is checked 
with $\chi$ (\cref{lem:euler}). For $A\in\A$, a decomposition as in \cref{lem:splitting} is fixed, and morphisms are 
transported along it.

\begin{proposition}\label{prop:precover}
Every $A\in\A$ has a special $\I$-precover.
\end{proposition}
\begin{proof}
Write $A\cong L\oplus U\oplus Q$ with $L=\Ho^{\le0}(A)$, $U=\Ho^{\ge1}(A)$ and $Q$ contractible. Define
\[
 E_A=L\oplus P(U)\oplus U[-1]\oplus Q,\qquad
 p_A=1_L\oplus q_U\oplus0\oplus1_Q:E_A\longrightarrow L\oplus U\oplus Q\cong A.
\]
By \eqref{eq:cones}, $p_A$ is degreewise surjective with kernel $T_A\cong U[-1]\oplus U[-1]$, 
one copy being $\ker q_U$. This gives a short exact sequence
\begin{equation}\label{eq:precover}
 0\longrightarrow T_A\longrightarrow E_A\xrightarrow{\ p_A\ }A\longrightarrow0.
\end{equation}
The complex $U[-1]$ has zero differential and $H^n(U[-1])=H^{n-1}(U)$, which vanishes unless $n\ge2$. 
Thus $\chi(T_A)=-2\chi(U)$ is even and $T_A\in\T$. 
Also $\chi(E_A)=\chi(L)-\chi(U)=\chi(A)-2\chi(U)$ is even, so $E_A\in\A$, and \eqref{eq:precover} is a conflation in $\A$.

Since $P(U)$ and $Q$ are contractible, $\Ho(E_A)\cong L\oplus U[-1]$, and $p_A$ induces the identity on $L$ and 
zero on $U[-1]$. As $L$ is concentrated in degrees $\le0$, we get $H^n(p_A)=0$ for $n\ge1$, that is, $p_A\in\I$. 
Since $1_{T_A}\in\J$, \cref{lem:object-special}(i) shows that $p_A$ is a special $\I$-precover.
\end{proof}

\begin{proposition}\label{prop:preenvelope}
Every $A\in\A$ has a special $\J$-preenvelope.
\end{proposition}
\begin{proof}
It is clear by the ideal version of Salce's lemma, see \cite[Theorem 18]{FGHT}.
However, here we give the similar construction as in \cref{prop:precover}. 
Write $A\cong V\oplus W\oplus Q$ with $V=\Ho^{\le1}(A)$, $W=\Ho^{\ge2}(A)$ and $Q$ contractible. Then define
\[
 D_A=\Cone(1_V)\oplus W\oplus V[1]\oplus Q,\qquad
 j_A=\iota_V\oplus1_W\oplus0\oplus1_Q:A\cong V\oplus W\oplus Q\longrightarrow D_A,
\]
where the zero component maps into $V[1]$. Hence we have the following conflation in $\A$
\begin{equation}\label{eq:preenvelope}
 0\longrightarrow A\xrightarrow{\ j_A\ }D_A\longrightarrow F_A\longrightarrow0,
\end{equation}
where $F_A\cong V[1]\oplus V[1]$.
It is clear that $1_{F_A}\in\I$ and $j_A\in\J$. Then \cref{lem:object-special}(ii) shows that $j_A$ is a special $\J$-preenvelope.
\end{proof}

\begin{corollary}\label{cor:ideal-complete}
The ideal cotorsion pair $(\I,\J)$ is complete.
\end{corollary}

\begin{remark}\label{rem:repair}
In \eqref{eq:precover} and \eqref{eq:preenvelope} the kernel and cokernel lie in the required object classes, but the middle terms in general do not. The extra summand $U[-1]$ in $E_A$ repairs the parity of $E_A$ at the cost of cohomology in degrees $\ge2$, which is forbidden in $\F$; the summand $V[1]$ in $D_A$ plays the dual role. The next section shows that this failure is unavoidable when the corresponding truncation has odd total cohomology dimension.
\end{remark}

\section{Object completeness and \texorpdfstring{$t$}{t}-structures}\label{sec:failure}

The ideal cotorsion pair $(\I,\J)$ is complete by
\cref{cor:ideal-complete}. We now determine when its induced
object cotorsion pair admits special approximations.

\subsection{The parity criterion}\label{sec:object-failure}

We use the standard cohomological notation $\K^{\le r}$,
$\K^{\ge r}$ and truncations in $\K\simeq D^b(\Vect_k)$
\cite[Section~1.3]{BBD}.
By \cref{lem:splitting}, $\tau^{\le r}X\cong\Ho^{\le r}(X)$
and $\tau^{\ge r}X\cong\Ho^{\ge r}(X)$.
Comparing truncation triangles with special object approximations
gives the following parity criterion.

\begin{proposition}\label{prop:criterion}
Let $A\in\A$.
\begin{enumerate}
\item $A$ has a special $\F$-precover if and only if $\beta(\Ho^{\le0}(A))$ is even.
\item $A$ has a special $\T$-preenvelope if and only if $\beta(\Ho^{\ge2}(A))$ is even.
\end{enumerate}
\end{proposition}
\begin{proof}
(i) A conflation $0\to T\to F\to A\to0$ with $F\in\F$ and $T\in\T$ gives,
via \eqref{eq:stable-category}, a triangle
$F\to A\to T[1]\xrightarrow{+}$ in $\K$ with $F\in\K^{\le0}$ and
$T[1]\in\K^{\ge1}$. Hence $F\cong\tau^{\le0}A\cong\Ho^{\le0}(A)$, and
$\beta(\Ho^{\le0}(A))=\beta(F)$ is even.

Conversely, write $A\cong L\oplus U\oplus Q$ as in \cref{prop:precover}.
If $\beta(L)$ is even, then so is $\beta(U)=\beta(A)-\beta(L)$. Hence
$U[-1]\in\T$ and $L\oplus P(U)\oplus Q\in\F$, and
\[
 0\longrightarrow U[-1]\longrightarrow L\oplus P(U)\oplus Q\xrightarrow{1_L\oplus q_U\oplus1_Q}A\longrightarrow0
\]
is a special $\F$-precover.

The proof of (ii) is dual, using \cref{prop:preenvelope} instead of
\cref{prop:precover}.
\end{proof}

Both parity conditions already fail for the following object.

\begin{corollary}\label{cor:failure}
Let $M=\stalk{0}{k}\oplus\stalk{2}{k}$. Then $M\in\A$, but $M$ has
neither a special $\F$-precover nor a special $\T$-preenvelope. In
particular, the cotorsion pair $(\F,\T)$ is neither special precovering
nor special preenveloping.
\end{corollary}
\begin{proof}
We have $\beta(M)=2$, while $\Ho^{\le0}(M)=\stalk{0}{k}$ and
$\Ho^{\ge2}(M)=\stalk{2}{k}$ both have $\beta=1$. Apply
\cref{prop:criterion}.
\end{proof}

\begin{proof}[Proof of \cref{thm:main}]
The properties of $\A$ are \cref{prop:category,prop:frobenius}. The
statements about $(\I,\J)$ are \cref{cor:pair,cor:ideal-complete}. The
failure of object completeness, with the witness $M$, is
\cref{cor:failure}.
\end{proof}

The witness $M$ is also relevant to the following criterion.

\begin{remark}\label{rem:STWZ}
By \cite[Corollary~1.3, $(1)\Leftrightarrow(3)$]{STWZ}, object
completeness is equivalent to the existence of a special
$\T$-preenvelope for every object of $\smd(\F\oplus\T)$, the class of
direct summands in $\A$ of objects $F\oplus T$ with $F\in\F$ and
$T\in\T$. Our $M$ lies in this class, since
$M\oplus M\cong\stalk{0}{k^2}\oplus\stalk{2}{k^2}\in\F\oplus\T$. Thus
\cref{cor:failure} also exhibits a witness for this criterion.
\end{remark}

\subsection{Special approximations in the stable category}\label{sec:stable}
\leavevmode\par
We interpret this failure through the correspondence between
cotorsion pairs and $t$-structures established by Saor\'{\i}n and
\v{S}\v{t}ov\'{\i}\v{c}ek \cite[Section~3]{SS}.

For the general results below, let $(\C,\E)$ be a Frobenius exact
category and fix an arbitrary ideal cotorsion pair $(\I,\J)$ of
object ideals in $\C$.
Let $\mathcal P$ be the projective--injective class of $\C$, and let
$\sC=\C/\mathcal P$ be its stable category, with suspension
$[1]=\Omega^{-1}$.

We use $[1]$ throughout.
We write ${}^{\perp_{\Hom}}\X$ and $\X^{\perp_{\Hom}}$ for
orthogonals with respect to $\sHomC$, and use suspended
subcategories and $t$-structures as in
\cite[Definitions~3.2 and~3.6]{SS}.
Thus a $t$-structure $(\X,\Y)$ has decomposition triangles
$X[1]\to M\to Y\xrightarrow{+}$ with $X\in\X$ and $Y\in\Y$.
For full subcategories $\X,\Y$ of $\sC$, write $\X*\Y$ for the objects
admitting a triangle $X\to M\to Y\xrightarrow{+}$ with $X\in\X$ and
$Y\in\Y$.

Put $\F=\Ob(\I)$ and $\T=\Ob(\J)$. Note that $(\F,\T)$ is a
cotorsion pair by \cref{lem:object-pair}. By \cite[Proposition~3.3]{SS}, $(\F,\T[1])$ is a
Hom-orthogonal pair in $\sC$, then
$\F$ and $\T$ are closed under isomorphisms, extensions
and direct summands in $\sC$.
We use the same notation for these classes and their images in $\sC$.

Following \cite[Section~3.1]{SS}, we say that the cotorsion pair $(\F,\T)$ in $\C$ is 
\emph{cohereditary} if $\F[1]\subseteq\F$. This is equivalent to
$\sHomC(F,T)=0$ for all $F\in\F$ and $T\in\T$, and to
$\T[-1]\subseteq\T$. Special approximations of a single object have the 
following description.

\begin{lemma}[{\cite[Proposition~3.3]{SS}}]\label{lem:stable-approx}
Let $A\in\C$. Then $A$ has a special $\F$-precover if and only if
$A\in\F*\T[1]$, and a special $\T$-preenvelope if and only if
$A[1]\in\F*\T[1]$. In particular, $(\F,\T)$ is complete if and only if
$\sC=\F*\T[1]$.
\end{lemma}

For cohereditary pairs, completeness is equivalent to
$(\F[-2],\T)$ being a $t$-structure
\cite[Proposition~3.9]{SS}. Note that $\I$ is precovering 
if and only if $\F$ is precovering in $\C$ by \cite[Lemma~3.2 (1)]{SWZ}. By a simple
verification, it follows that $\F$ is precovering in $\sC$ if $\I$ is precovering.
When the stable category is idempotent complete,
\cite[Proposition~3.11]{SS} supplies the approximation triangles required
by \cref{lem:stable-approx} for cohereditary cotorsion pairs. This gives the
following positive answer to \cref{qu:original}.


\begin{corollary}\label{cor:positive}
Let $\C$ be a Frobenius exact category whose stable category $\sC$ is
idempotent complete, and let $(\I,\J)$ be an ideal cotorsion pair of
object ideals in $\C$ such that $(\F,\T)=(\Ob(\I),\Ob(\J))$ is
cohereditary. If $(\I,\J)$ is complete, equivalently if $\I$ is
precovering \cite[Corollary~1.2]{STWZ}, then $(\F,\T)$ is complete.
\end{corollary}
\begin{proof}
We note that $\F$ is closed under extensions and direct summands
in $\sC$, and $\F[1]\subseteq\F$, thus $\F$ is suspended. Since $\I$ is precovering,
then $\F$ is precovering in $\sC$. Notice that $\sC$ is idempotent complete,
so $(\F,\F^{\perp_{\Hom}}[1])$ is a $t$-structure
by \cite[Proposition~3.11]{SS}. Shifting its decomposition triangles by $[-1]$
gives $\sC=\F*\F^{\perp_{\Hom}}=\F*\T[1]$. 
Now the assertion follows from \cref{lem:stable-approx}.
\end{proof}

For the object-wise construction, we first replace a special
ideal precover by a conflation with kernel in $\T$.

\begin{lemma}[{\cite[Proposition~25]{FGHT}}]\label{lem:reduced-precover}
	If $A\in\C$ has a special $\I$-precover, then there is a conflation
	$0\to T_0\xrightarrow{\iota}G\xrightarrow{q}A\to0$ with $T_0\in\T$ and
	$q\in\I$.
\end{lemma}

From such a conflation we now construct the idempotent.

\begin{proposition}\label{prop:t-type-positive}
Assume that the cotorsion pair $(\F,\T)$ is cohereditary, and let
$0\to T_0\xrightarrow{\iota}G\xrightarrow{q}A\to0$ be a conflation with
$T_0\in\T$ and $q\in\I$.
\begin{enumerate}
\item There are $F_0\in\F$ and morphisms $a:G\to F_0$, $v:F_0\to G$ and
$t:G\to T_0$ with $1_G=va+\iota t$. For any such choice,
$\varepsilon=\ul{\iota t}\in\operatorname{End}_{\sC}(G)$ is idempotent with
$1-\varepsilon=\ul{va}$.
\item Idempotents $\varepsilon,\varepsilon'$ obtained from two choices in
\textup{(i)} satisfy $\varepsilon\varepsilon'=\varepsilon$ and
$\varepsilon'\varepsilon=\varepsilon'$; in particular, $\varepsilon$
splits in $\sC$ if and only if $\varepsilon'$ does.
\item $\varepsilon$ splits in $\sC$ if and only if $A$ has a special $\F$-precover.
\end{enumerate}
\end{proposition}
\begin{proof}
(i) Note that $\I=\langle\F\rangle$, then suppose that $q=ba$ with 
$a:G\to F_0$, $b:F_0\to A$ and $F_0\in\F$. As
$\Ext_{\C}(F_0,T_0)=0$, there is $v:F_0\to G$ with $b=qv$. Then
$q(1_G-va)=0$, so $1_G-va=\iota t$ for some $t:G\to T_0$. Composing with
$t$ gives $t=(tv)a+(t\iota)t$, and $\ul{tv}=0$ by coheredity. Hence
$\ul t=\ul{t\iota}\,\ul t$, so $\varepsilon^2=\varepsilon$.

(ii) If $\varepsilon'=\ul{\iota t'}$, then
$\varepsilon'(1-\varepsilon)=\ul{\iota(t'v)a}=0$, so
$\varepsilon'=\varepsilon'\varepsilon$. Symmetrically we have
$\varepsilon=\varepsilon\varepsilon'$. If $\varepsilon=ip$ with $pi=1$,
then $\varepsilon'=(\varepsilon'i)(p\varepsilon)$ and
$(p\varepsilon)(\varepsilon'i)=p\varepsilon i=1$.

(iii) If $\varepsilon$ splits, then so does $1-\varepsilon$, as $\sC$ is
triangulated. Then we have $G\cong G_1\oplus G_2$ in $\sC$ with the following diagram
$$\xymatrix{
	G_1 \ar@<0.5ex>[r]^{i_1} & G \ar@<0.5ex>[r]^{p_2} \ar@<0.5ex>[l]^{p_1} & G_2 \ar@<0.5ex>[l]^{i_2}
}$$
where $G_1$ and $G_2$ are the images of 
$1-\varepsilon=\ul{va}=i_1p_1$ and $\varepsilon=\ul{\iota t}=i_2p_2$. Then we have
$(p_1\ul{v})(\ul{a}i_1)=1_{G_1}$ and $(p_2\ul{\iota})(\ul{t}i_2)=1_{G_2}$.
It follows that $G_1$ and $G_2$ are direct summands of $F_0$ and $T_0$ in $\sC$, respectively.
So $G_1\in\F$ and $G_2\in\T$ and then $G\in\F*\T$. 
The triangle $T_0\to G\to A\xrightarrow{+}$ gives
\[
 A\in(\F*\T)*\T[1]=\F*(\T*\T[1])\subseteq\F*\T[1],
\]
since $\T\subseteq\T[1]$ by coheredity and $\T[1]$ is closed under
extensions. Now apply \cref{lem:stable-approx}.

Conversely, let
$0\to T_1\xrightarrow{j}F_1\xrightarrow{p}A\to0$ be a special
$\F$-precover. The factorization $q=ba$ from (i) and the vanishing of
$\Ext_{\C}(F_0,T_1)$ and $\Ext_{\C}(F_1,T_0)$ give morphisms
$u:G\to F_1$ and $w:F_1\to G$ with $pu=q$ and $qw=p$.
Since $p(1_{F_1}-uw)=0$, we have $1_{F_1}-uw=js$ for some
$s:F_1\to T_1$. Coheredity gives $\ul s=0$, hence
$\ul u\,\ul w=1_{F_1}$, so $\ul{wu}$ is a split idempotent.
Also $q(1_G-wu)=0$, so $1_G-wu=\iota t'$ for some $t':G\to T_0$.
Taking $F_0=F_1$, $a=u$ and $v=w$ gives the associated idempotent
$\varepsilon'=\ul{\iota t'}=1-\ul{wu}$, which splits because $\sC$
is triangulated. By (ii), $\varepsilon$ splits as well.
\end{proof}

By \cref{lem:stable-approx}, $A$ has a special $\T$-preenvelope if and only
if $A[1]$ has a special $\F$-precover, so
\cref{prop:t-type-positive} applied to $A[1]$ also governs special
$\T$-preenvelopes.


We return to the parity example and interpret \cref{cor:failure} in
terms of $t$-structures and stable idempotents. Take $\C=\A$, with
$\mathcal P$ the contractible complexes and
$(\I,\J)$, $(\F,\T)$ as in \cref{sec:pair}. By \cref{lem:homotopy},
$\sHom(F,T)\cong\prod_n\Hom_k(H^n(F),H^n(T))=0$ for $F\in\F$ and
$T\in\T$, so $(\F,\T)$ is cohereditary. By
\cref{cor:ideal-complete},
$\F$ is precovering and $\T$ is preenveloping in $\A$.

Under \eqref{eq:stable-category},
$\F=\mathcal S\cap\K^{\le0}$ and $\T=\mathcal S\cap\K^{\ge2}$, so
$\F$ and $\T[1]=\mathcal S\cap\K^{\ge1}$ are the traces on $\mathcal S$ of the
standard aisle $\K^{\le0}$ and its right orthogonal $\K^{\ge1}$. The ideal cotorsion
pair is complete, but $\stab\neq\F*\T[1]$:
by uniqueness of truncation triangles, $M=\stalk{0}{k}\oplus\stalk{2}{k}$ 
(see \cref{cor:failure}) would lie in $\F*\T[1]$ only if the outer terms of
$\stalk{0}{k}\to M\to\stalk{2}{k}\xrightarrow{+}$ lay in
$\mathcal S$.

We now specialize \cref{prop:t-type-positive} to the parity example
of \cref{sec:pair}, with $\C=\A$ and $\I,\J,\F,\T$ as defined there.
Its splitting criterion becomes an explicit cohomological parity
condition.

\begin{corollary}\label{cor:parity-idempotent}
Let $A\in\A$, let $0\to T_0\xrightarrow{\iota}G\xrightarrow{q}A\to0$ be a
conflation in $\A$ with $T_0\in\T$ and $q\in\I$ \textup{(}for instance
\eqref{eq:precover}\textup{)}, and let $\varepsilon$ be as in
\cref{prop:t-type-positive}. Then $H^1(G)=0$, and under
\eqref{eq:stable-category} $\varepsilon$ is the projection of
$\Ho(G)=\Ho^{\le0}(G)\oplus\Ho^{\ge2}(G)$ onto $\Ho^{\ge2}(G)$. Moreover,
the following are equivalent:
\begin{enumerate}
\item $\varepsilon$ splits in $\stab$;
\item $\beta(\Ho^{\le0}(A))$ is even;
\item $A$ has a special $\F$-precover.
\end{enumerate}
\end{corollary}
\begin{proof}
It is clear by \cref{prop:criterion,prop:t-type-positive}. 
\end{proof}

\begin{remark}[The example revisited]\label{rem:witness-idempotent}
Let $M=\stalk{0}{k}\oplus\stalk{2}{k}$ and let $A=M$ or $M[1]$.
For any conflation $0\to T_0\xrightarrow{\iota}G\xrightarrow{q}A\to0$
in $\A$ with $T_0\in\T$ and $q\in\I$, the associated idempotent
of \cref{prop:t-type-positive} does not split in $\stab$.
Consequently, $M$ has neither a special $\F$-precover nor a special
$\T$-preenvelope. In particular, the idempotent completeness hypothesis
in \cref{cor:positive} cannot be omitted. 
Indeed, since $M[1]=\stalk{-1}{k}\oplus\stalk{1}{k}$, we have
\[
 \beta(\Ho^{\le0}(M))=\beta(\Ho^{\le0}(M[1]))=1.
\]
By \cref{cor:parity-idempotent}, none of the associated idempotents
splits. The same proposition, applied to \eqref{eq:precover} for $M$
and $M[1]$, shows that neither has a special $\F$-precover.
Now \cref{lem:stable-approx} implies that $M$ has no special
$\T$-preenvelope.

For the conflation \eqref{eq:precover} with $A=M$, the middle term is
\[
 E_M=\stalk{0}{k}\oplus P(\stalk{2}{k})\oplus\stalk{3}{k}.
\]
As $P(\stalk{2}{k})$ is contractible,
\cref{cor:parity-idempotent} identifies $\varepsilon$ with the
projection onto $\stalk{3}{k}$ under \eqref{eq:stable-category}.
This image has odd total cohomology dimension and does not belong
to $\mathcal S$. So in this example, all the conditions in \cref{cor:positive} 
are satisfied except idempotent completeness.

\end{remark}

Finally, the parity obstruction disappears after idempotent
completion. Since $\D$ and $\K$ are idempotent complete, the same
doubling argument identifies them with the idempotent completions of
$\A$ and $\stab\simeq\mathcal S$, respectively \cite{Buhler,BS}.
In $\D$, all truncations are available, and the cohomological
support classes in degrees $\le0$ and $\ge2$ form a complete
cotorsion pair by the same Ext formula and cone sequences.
This leads to the following question.

\begin{question}\label{qu:idempotent-complete}
Let $(\I,\J)$ be a complete ideal cotorsion pair in an idempotent
complete exact category. If both $\I$ and $\J$ are object ideals,
must the cotorsion pair $(\Ob(\I),\Ob(\J))$ be complete?
\end{question}

\subsection*{Declaration of AI use}
The counterexample presented in this paper was developed with the assistance of OpenAI's GPT-6 Astra. The authors organized the mathematical arguments and wrote the manuscript with the assistance of Anthropic's Claude. The authors independently verified all mathematical arguments and references and take full responsibility for the mathematical content and the final version of the manuscript.

\end{document}